\documentclass[draftclsnofoot,onecolumn]{IEEEtran}
\usepackage{amsthm}

\newtheorem{lemma}{Lemma}
\theoremstyle{definition}
\newtheorem{definition}{Definition}

\usepackage{subfigure}
\usepackage{float}
\usepackage{amsfonts}
\usepackage{verbatim}
\usepackage{amssymb}

\usepackage{booktabs}

\usepackage{algorithm}
\usepackage{algorithmic}

\ifCLASSINFOpdf
\else
   \usepackage[dvips]{graphicx}
   \graphicspath{{../eps/}}
   \DeclareGraphicsExtensions{.eps}
\fi
\ifCLASSOPTIONcompsoc
 \usepackage[caption=false,font=normalsize,labelfont=sf,textfont=sf]{subfig}
\else
 \usepackage[caption=false,font=footnotesize]{subfig}
\fi
\begin{document}
%
\title{A non-convex approach to low-rank and sparse matrix decomposition}
%
%
%

\author{Angang~Cui,
        Meng~Wen,
        Haiyang~Li
        and~Jigen Peng
\thanks{A. Cui is with the School of Mathematics and Statistics, Xi'an Jiaotong University, Xi'an, 710049, China. e-mail: (cuiangang@163.com).}
\thanks{M. Wen is with the School of Science, Xi'an Polytechnic University, Xi'an, 710048, China. e-mail: (wen5495688@163.com).}
\thanks{H. Li and J. Peng are with the School of Mathematics and Information Science, Guangzhou University, Guangzhou, 510006, China. e-mail: (fplihaiyang@126.com; jgpengxjtu@126.com).}
\thanks{This work was supported by the National Natural Science Foundations of China (11771347, 91730306, 41390454, 11271297) and the Science Foundations of Shaanxi Province of
China (2016JQ1029, 2015JM1012).}
\thanks{Manuscript received, ; revised , .}}

%
%

\markboth{Journal of \LaTeX\ Class Files,~Vol.~, No.~, ~}%
{Shell \MakeLowercase{\textit{et al.}}: Bare Demo of IEEEtran.cls for IEEE Journals}
%



\maketitle

\begin{abstract}
In this paper, we develop a nonconvex approach to the problem of low-rank and sparse matrix decomposition. In our nonconvex method, we replace the rank function and
the $l_{0}$-norm of a given matrix with a non-convex fraction function on the singular values and the elements of the matrix respectively. An alternative direction
method of multipliers algorithm is utilized to solve our proposed nonconvex problem with the nonconvex fraction function penalty. Numerical experiments on some low-rank
and sparse matrix decomposition problems show that our method performs very well in recovering low-rank matrices which are heavily corrupted by large sparse errors.
\end{abstract}

\begin{IEEEkeywords}
Low-rank and sparse matrix decomposition, Nonconvex fraction function, Alternative direction method of multipliers algorithm.
\end{IEEEkeywords}

%
\IEEEpeerreviewmaketitle

\section{Introduction} \label{section1}

In many scientific and engineering applications, such as separating foregrounds and backgrounds from videos \cite{Wright2009,Liu2017}, shadows and specularities removing in
face images \cite{Candes2011}, Latent semantic indexing \cite{Candes2011, Min2010}, image processing \cite{Zhang2011,Wu2018}, the observed data matrix $M\in \mathbb{R}^{m\times n}$
can naturally be decomposed into a low-rank matrix $L\in \mathbb{R}^{m\times n}$ and a corrupted sparse matrix $S\in \mathbb{R}^{m\times n}$ with the arbitrarily large elements.
That is, the observed data matrix $M\in \mathbb{R}^{m\times n}$ can be decomposed as
\begin{equation}\label{equ1}
M=L+S,
\end{equation}
where $L\in \mathbb{R}^{m\times n}$ is the low-rank matrix, and $S\in \mathbb{R}^{m\times n}$ is the sparse error matrix with the arbitrarily large elements. Without loss of generality,
we assume $m\geq n$ in throughout this paper. In mathematics, decomposing the observed data matrix $M\in \mathbb{R}^{m\times n}$ in to a sum of the low-rank matrix $L\in \mathbb{R}^{m\times n}$
and sparse matrix $S\in \mathbb{R}^{m\times n}$ can be described as the following minimization problem \cite{Peng2013}:
\begin{equation}\label{equ2}
\min_{L,S\in \mathbb{R}^{m\times n}}\ \mathrm{rank}(L)+\lambda \|S\|_{0},\ \ s.t. \ \ M=L+S,
\end{equation}
where $\lambda>0$ is a tuning parameter, $\mathrm{rank}(L)$ represents the rank of matrix $L$, and $\|S\|_{0}$ is the $l_{0}$-norm of the matrix $S$, which counts
the number of nonzero elements of the matrix $S$. In general, problem (\ref{equ2}) is a challenging nonconvex optimization problem \cite{Peng2013} because of the
discrete natures of the rank function $\mathrm{rank}(L)$ and $l_{0}$-norm $\|S\|_{0}$.
Inspired by the good performance of the nonconvex fraction function in our latest work \cite{Li2019} in compressed sensing , in this paper, we replace the rank
function $\mathrm{rank}(L)$ and the $l_{0}$-norm $\|S\|_{0}$ in the nonconvex problem (\ref{equ2}) with the continuous promoting low rank non-convex function
\begin{equation}\label{equ3}
F_{a_{1}}(\sigma(L))=\sum_{i\in[n]}\rho_{a_{1}}(\sigma_{i}(L))=\sum_{i\in[n]}\frac{a_{1}\sigma_{i}(L)}{a_{1}\sigma_{i}(L)+1}
\end{equation}
and the continuous promoting sparse non-convex function
\begin{equation}\label{equ4}
F_{a_{2}}(S)=\sum_{l\in[m],j\in[n]}\rho_{a_{2}}(S_{l,j})=\sum_{l\in[m],j\in[n]}\frac{a_{2}|S_{l,j}|}{a_{2}|S_{l,j}|+1}
\end{equation}
respectively, where the parameters $a_{1}, a_{2}\in(0,+\infty)$, $\sigma(L)$ is the vector of singular values of matrix $L$ arranged in descending order
and $\sigma_{i}(L)$ is the $i$-th largest element of $\sigma(L)$, $S_{l,j}$ is the element in the $i$-th row and $j$-th column of matrix $S$, and the
non-convex fraction function is defined as
\begin{equation}\label{equ5}
\rho_{a}(t)=\frac{a|t|}{a|t|+1}
\end{equation}
for all $a\in(0,+\infty)$ and $t\in \mathbb{R}$. With the change of parameter $a>0$, we have
\begin{equation}\label{equ6}
\lim_{a\rightarrow+\infty}\rho_{a}(t)=\lim_{a\rightarrow+\infty}\frac{a|t|}{a|t|+1}=\left\{
    \begin{array}{ll}
      0, & {\ \ \mathrm{if} \ t=0;} \\
      1, & {\ \ \mathrm{if} \ t\neq 0.}
    \end{array}
  \right.
\end{equation}
Then, the function $F_{a_{1}}(\sigma(L))$ interpolates the rank of matrix $L$:
\begin{equation}\label{equ7}
\lim_{a_{1}\rightarrow+\infty}F_{a_{1}}(\sigma(L))=\lim_{a_{1}\rightarrow+\infty}\sum_{i\in[n]}\rho_{a_{1}}(\sigma_{i}(L))=\mathrm{rank}(L)
\end{equation}
and the function $F_{a_{2}}(S)$ interpolates the $l_{0}$-norm of matrix $S$:
\begin{equation}\label{equ8}
\lim_{a_{2}\rightarrow+\infty}F_{a_{2}}(S)=\lim_{a_{2}\rightarrow+\infty}\sum_{l\in[m],j\in[n]}\rho_{a_{2}}(S_{l,j})=\|S\|_{0}.
\end{equation}
By above transformation, we can state the following minimization problem
\begin{equation}\label{equ9}
\min_{L,S\in \mathbb{R}^{m\times n}}\ F_{a_{1}}(\sigma(L))+\lambda F_{a_{2}}(S),\ \ s.t. \ \ M=L+S
\end{equation}
as the approximation for the nonconvex problem (\ref{equ2}).

One of the advantages for problem (\ref{equ9}) is that it will be flexible compared with the nonconvex problem (\ref{equ2}) with the change of the parameters $a_{1}, a_{2}\in(0,+\infty)$.
Unfortunately, the non-convex constrained problem (\ref{equ9}) is still computationally harder to solve due to the nonconvexity of the nonconvex fraction function. Usually, we
consider its augmented Lagrange version leading to an optimization problem that can be easily solved. The augmented Lagrange version for the problem (\ref{equ9}) can be described as the
following minimization problem:
\begin{equation}\label{equ10}
\displaystyle\min_{L,S\in \mathbb{R}^{m\times n}}\ F_{a_{1}}(\sigma(L))+\lambda F_{a_{2}}(S)+\langle Y, M-L-S\rangle+\displaystyle \frac{\mu}{2}\|M-L-S\|_{F}^{2},
\end{equation}
where $Y\in \mathbb{R}^{m\times n}$ is the Lagrange multiplier matrix and $\mu>0$ is the penalty parameter .

The rest of this paper is organized as follows. In Section \ref{section2}, we summarize some preliminary results that will be used in this paper. In Section \ref{section3}, we use an
alternative direction method of multipliers algorithm to solve the the augmented Lagrange problem (\ref{equ10}). In Section \ref{section4}, we present some numerical experiments
on some low-rank and sparse matrix decomposition problems to demonstrate the performances of our method. Finally, some conclusion remarks are presented in Section \ref{section5}.

\section{Preliminaries} \label{section2}
In this section, we summarize some crucial lemmas and definitions that will be used in this paper.

\begin{lemma}\label{lem1}
Define a function of $\beta\in \mathbb{R}$ as
\begin{equation}\label{equ11}
f_{a, \tau}(\beta)=\frac{1}{2}(\beta-\gamma)^{2}+\tau\rho_{a}(\beta)
\end{equation}
where $\gamma\in \mathbb{R}$ and $\tau>0$, the proximal operator $h_{f_{a, \tau}}(\gamma):=\arg\min_{\beta\in \mathbb{R}}f_{a, \tau}(\beta)$ can be described as
\begin{equation}\label{equ12}
h_{f_{a, \tau}}(\gamma)=\left\{
    \begin{array}{ll}
      g_{a,\tau}(\gamma), & \ \ \mathrm{if} \ {|\gamma|> t_{a,\tau};} \\
      0, & \ \ \mathrm{if} \ {|\gamma|\leq t_{a,\tau}.}
    \end{array}
  \right.
\end{equation}
where $g_{a,\tau}(\gamma)$ is defined as
\begin{equation}\label{equ13}
g_{a,\tau}(\gamma)=\mathrm{sign}(\gamma)\bigg(\frac{\frac{1+a|\gamma|}{3}(1+2\cos(\frac{\phi(\gamma)}{3}-\frac{\pi}{3}))-1}{a}\bigg),
\end{equation}
$$\phi(\gamma)=\arccos\Big(\frac{27\tau a^{2}}{2(1+a|\gamma|)^{3}}-1\Big),$$
and the threshold value $t_{a,\tau}$ satisfies
\begin{equation}\label{equ14}
t_{a,\tau}=\left\{
    \begin{array}{ll}
      \tau a, & \ \ \mathrm{if} \ {\tau\leq \frac{1}{2a^{2}};} \\
      \sqrt{2\tau}-\frac{1}{2a}, & \ \ \mathrm{if} \ {\tau>\frac{1}{2a^{2}}.}
    \end{array}
  \right.
\end{equation}
\end{lemma}

\begin{proof}
The proof of Lemma \ref{lem1} using the Cartan's root-finding formula expressed in terms of hyperbolic functions (see \cite{xing2003}). It is also similar to the proof of Lemma 10 in \cite{Li2019},
so it is omitted here.
\end{proof}

\begin{definition}\label{de1}
(Vector thresholding operator) For any $a>0$, $\tau>0$ and $x\in \mathbb{R}^{n}$, we define the vector thresholding operator $H_{f_{a,\tau}}$ on vector $x$ as
\begin{equation}\label{equ15}
\mathcal{H}_{f_{a, \tau}}(x)=(h_{f_{a, \tau}}(x_{1}), h_{f_{a, \tau}}(x_{2}),\cdots, h_{f_{a, \tau}}(x_{n}))^{\top},
\end{equation}
where the proximal mapping operator $h_{f_{a, \tau}}$ is defined in Lemma \ref{lem1}, and $x_{i}$ is the $i$-th element of the vector $x\in \mathbb{R}^{n}$.
\end{definition}

\begin{definition}\label{de2}
(Matrix element thresholding operator) For any matrix $B\in \mathbb{R}^{m\times n}$, we define the matrix element thresholding operator $\mathcal{D}_{a, \tau}$ on matrix $B$ as
\begin{equation}\label{equ16}
\mathcal{D}_{a, \tau}(B)=[h_{f_{a, \tau}}(B_{l,j})],
\end{equation}
where $l\in[m]$, $j\in[n]$, and the proximal operator $h_{f_{a, \tau}}$ is defined in Lemma \ref{lem1}.
\end{definition}

The matrix element thresholding operator $\mathcal{D}_{a, \tau}$ defined in Definition \ref{de2} simply applies the proximal operator $h_{f_{a, \tau}}$ to the
elements of a matrix. If many of the elements of matrix $B$ are below the threshold value $t_{a,\tau}$, the matrix element thresholding operator $\mathcal{D}_{a, \tau}$
effectively shrinks them towards zero, and the matrix $\mathcal{D}_{a, \tau}(B)$ is a sparse matrix. Similarly, we can also define the following matrix singular value
thresholding operator $\mathcal{G}_{a,\tau}$ whcih applies the proximal operator $h_{f_{a, \tau}}$ to the singular values of a matrix.

\begin{definition}\label{de3}
(Matrix singular value thresholding operator) Suppose the matrix $N\in \mathbb{R}^{m\times n}$ admits a singular value decomposition (SVD) as
$$N=U\left[
       \begin{array}{c}
         \mathrm{Diag}(\sigma(N)) \\
         \mathbf{0}_{(m-n)\times n} \\
       \end{array}
     \right]
     V^{\top},$$
where $U$ is a $m\times m$ unitary matrix, $V$ is a $n\times n$ unitary matrix, $\mathrm{Diag}(\sigma(N))$ is a $n\times n$ diagonal matrix and $\mathbf{0}_{(m-n)\times n}$ is a
$(m-n)\times n$ zero matrix. We define the matrix singular value thresholding operator $\mathcal{G}_{a,\tau}$ on matrix $N$ as
\begin{equation}\label{equ17}
\mathcal{G}_{a,\tau}(N)=U\left[
                           \begin{array}{c}
                             \mathrm{Diag}(\mathcal{H}_{f_{a,\tau}}(\sigma(N))) \\
                             \mathbf{0}_{(m-n)\times n} \\
                           \end{array}
                         \right]
V^{\top},
\end{equation}
where the vector thresholding operator $\mathcal{H}_{f_{a,\tau}}$ is defined in Definition \ref{de1}.
\end{definition}

By Definition \ref{de3}, we know that if there are some nonzero singular values of matrix $N$ are below the threshold value $t_{a,\tau}$, the rank of matrix $\mathcal{G}_{a,\tau}(N)$
must be lower than the rank of matrix $N$. Furthermore, according to Lemma \ref{lem1} and [\cite{Lu2015}, Proposition 2.1], we can get the following result.

\begin{lemma}\label{lem2}
Let
$$N=U\left[
       \begin{array}{c}
         \mathrm{Diag}(\sigma(N)) \\
         \mathbf{0}_{(m-n)\times n} \\
       \end{array}
     \right]
     V^{\top}$$
be the singular value decomposition of matrix $N\in \mathbb{R}^{m\times n}$ and
$$\mathcal{G}_{a,\tau}(N)=U\left[
                           \begin{array}{c}
                             \mathrm{Diag}(\mathcal{H}_{f_{a,\tau}}(\sigma(N))) \\
                             \mathbf{0}_{(m-n)\times n} \\
                           \end{array}
                         \right]
V^{\top}.$$
Then
\begin{equation}\label{equ19}
\mathcal{G}_{a,\tau}(N):=\arg\min_{Z\in \mathbb{R}^{m\times n}}\Big\{\frac{1}{2}\|Z-N\|_{F}^{2}+\tau P_{a}(\sigma(Z))\Big\}.
\end{equation}
\end{lemma}

\section{The algorithm for solving the augmented Lagrange problem (\ref{equ10})}\label{section3}

In this section, we use an alternative direction method of multipliers (ADMM) algorithm \cite{Gabay1976,Boyd2010} to solve our augmented Lagrange problem (\ref{equ10}). The ADMM algorithm can break
the problem (\ref{equ10}) into two smaller sub-problems, each of which is easy to handle. Now, we first review the basic process of the ADMM algorithm.

\subsection{The ADMM algorithm}\label{section3-1}

In general, the ADMM algorithm \cite{Gabay1976,Boyd2010} solves the problem in the form
\begin{equation}\label{equ19}
\min\ \tilde{f}(x)+\tilde{g}(z),\ \ s.t. \ \ Ax+Bz=c,
\end{equation}
where $x\in \mathbb{R}^{n}$, $z\in \mathbb{R}^{m}$, $A\in \mathbb{R}^{p\times n}$, $B\in \mathbb{R}^{p\times m}$, $c\in \mathbb{R}^{p}$, $\tilde{f}: \mathbb{R}^{n}\rightarrow \mathbb{R}$,
$\tilde{g}: \mathbb{R}^{m}\rightarrow \mathbb{R}$. The augmented Lagrange function for the problem (\ref{equ19}) is defined as
\begin{equation}\label{equ20}
L_{\mu}(x,z,y)=\tilde{f}(x)+\tilde{g}(z)+y^{\top}(Ax+Bz-c)+\displaystyle \frac{\mu}{2}\|Ax+Bz-c\|_{2}^{2},
\end{equation}
where $y\in \mathbb{R}^{p}$ is the Lagrange multiplier and $\mu>0$ is the penalty parameter.

Given $z^{0}\in \mathbb{R}^{m}$ and $y^{0}\in \mathbb{R}^{p}$, for $k=0, 1,2,\cdots$, the ADMM algorithm consists the iterations
\begin{equation}\label{equ21}
\left\{
  \begin{array}{ll}
    \hbox{$x^{k+1}=\displaystyle\arg\min_{x\in \mathbb{R}^{n}}L_{\mu}(x, z^{k},y^{k})$;} \\
    \hbox{$z^{k+1}=\displaystyle\arg\min_{z\in \mathbb{R}^{m}}L_{\mu}(x^{k+1}, z,y^{k})$;} \\
    \hbox{$y^{k+1}=\displaystyle y^{k}+\mu(Ax^{k+1}+Bz^{k+1}-c)$.}
  \end{array}
\right.
\end{equation}
We can see that the ADMM algorithm consists of a $x$-minimization step, a $z$-minimization step, and a dual variable update. The dual variable $y$ update uses
a step size $\mu$.

ADMM algorithm is a simple and effective method for separable programming problems. Its greatest advantage is that it makes full use of the separability of the objective function,
decomposes the original problem into several alternating minimizer problems which are easier to obtain the optimal solution for analysis, and it is more suitable for large-scale
problems with a large number of variables in practical application.

\subsection{The ADMM algorithm for solving the augmented Lagrange problem (\ref{equ10})}\label{section3-1}

In this subsection, we process the ADMM algorithm to solve our augmented Lagrange problem (\ref{equ10}). In order to convenient, we write the augmented Lagrange function for the
problem (\ref{equ10}) as
\begin{equation}\label{equ22}
\begin{array}{llll}
L_{a_{1}, a_{2}, \mu}(L,S,Y)=F_{a_{1}}(\sigma(L))+\lambda F_{a_{2}}(S)+\langle Y, M-L-S\rangle+\displaystyle\frac{\mu}{2}\|M-L-S\|_{F}^{2}.
\end{array}
\end{equation}
Therefore, given $S^{0}\in \mathbb{R}^{m\times n}$ and $Y^{0}\in \mathbb{R}^{m\times n}$, for $k=0, 1,2,\cdots$, the ADMM algorithm for solving the augmented Lagrange
problem (\ref{equ10}) can be described as
\begin{equation}\label{equ23}
\left\{
  \begin{array}{ll}
    \hbox{$L^{k+1}=\displaystyle\arg\min_{L\in \mathbb{R}^{m\times n}}L_{a_{1},a_{2},\mu}(L, S^{k},Y^{k})$;} \\
    \hbox{$S^{k+1}=\displaystyle\arg\min_{S\in \mathbb{R}^{m\times n}}L_{a_{1},a_{2},\mu}(L^{k+1}, S,Y^{k})$;} \\
    \hbox{$Y^{k+1}=\displaystyle Y^{k}+\mu(M-L^{k+1}-S^{k+1})$.}
  \end{array}
\right.
\end{equation}
Combing the truth that
\begin{equation}\label{equ24}
\begin{array}{llll}
L^{k+1}&=&\displaystyle\arg\min_{L\in \mathbb{R}^{m\times n}}\ L_{a_{1},a_{2},\mu}(L,S^{k},Y^{k})\\
&=&\displaystyle\arg\min_{L\in \mathbb{R}^{m\times n}}\ F_{a_{1}}(\sigma(L))+\lambda F_{a_{2}}(S^{k})+\displaystyle\frac{\mu}{2}\|M-L-S^{k}+\mu^{-1}Y^{k}\|_{F}^{2}\\
&=&\displaystyle\arg\min_{L\in \mathbb{R}^{m\times n}}\ \mu^{-1}F_{a_{1}}(\sigma(L))+\lambda\mu^{-1} F_{a_{2}}(S^{k})+\displaystyle\frac{1}{2}\|M-L-S^{k}+\mu^{-1}Y^{k}\|_{F}^{2}\\
&=&\mathcal{G}_{a_{1},\mu^{-1}}(M-S^{k}+\mu^{-1}Y^{k})
\end{array}
\end{equation}
and
\begin{equation}\label{equ25}
\begin{array}{llll}
S^{k+1}&=&\displaystyle\arg\min_{S\in \mathbb{R}^{m\times n}}\ L_{a_{1},a_{2},\mu}(L^{k+1},S,Y^{k})\\
&=&\displaystyle\arg\min_{L\in \mathbb{R}^{m\times n}}\ F_{a_{1}}(\sigma(L^{k+1}))+\lambda F_{a_{2}}(S)+\displaystyle\frac{\mu}{2}\|M-L^{k+1}-S+\mu^{-1}Y^{k}\|_{F}^{2}\\
&=&\displaystyle\arg\min_{L\in \mathbb{R}^{m\times n}}\ \mu^{-1}F_{a_{1}}(\sigma(L^{k+1}))+\lambda\mu^{-1} F_{a_{2}}(S)+\displaystyle\frac{1}{2}\|M-L^{k+1}-S+\mu^{-1}Y^{k}\|_{F}^{2}\\
&=&\mathcal{D}_{a_{2},\lambda\mu^{-1}}(M-L^{k+1}+\mu^{-1}Y^{k}),
\end{array}
\end{equation}
then the analytical expression of (\ref{equ23}) can be expressed as
\begin{equation}\label{equ26}
\left\{
  \begin{array}{ll}
    \hbox{$L^{k+1}=\mathcal{G}_{a_{1},\mu^{-1}}(M-S^{k}+\mu^{-1}Y^{k})$;} \\
    \hbox{$S^{k+1}=\mathcal{D}_{a_{2},\lambda\mu^{-1}}(M-L^{k+1}+\mu^{-1}Y^{k})$;} \\
    \hbox{$Y^{k+1}=\displaystyle Y^{k}+\mu(M-L^{k+1}-S^{k+1})$,}
  \end{array}
\right.
\end{equation}
where $\mathcal{G}_{a_{1},\mu^{-1}}$ is obtained by replacing $a$ and $\tau$ with $a_{1}$ and $\mu^{-1}$ in $\mathcal{G}_{a,\tau}$, and $\mathcal{D}_{a_{2},\lambda\mu^{-1}}$ is obtained
by replacing $a$ and $\tau$ with $a_{2}$ and $\lambda\mu^{-1}$ in $\mathcal{D}_{a,\tau}$. The ADMM algorithm for solving the augmented Lagrange problem (\ref{equ10}) can be summarized in the following Algorithm \ref{alg1}.

\begin{algorithm}[h!]
\caption{: ADMM algorithm for solving the augmented Lagrange problem (\ref{equ10})}
\label{alg1}
\begin{algorithmic}
\STATE {\textbf{Initialize}: $S^{0}, Y^{0}\in \mathbb{R}^{m\times n}$, $a_{1}>0$, $a_{2}>0$, $\mu>0$, $\lambda>0$;}
\STATE {$k=0$;}
\STATE {\textbf{while} not converged \textbf{do}}
\STATE {1.\ $L^{k+1}=\mathcal{G}_{a_{1},\mu^{-1}}(M-S^{k}+\mu^{-1}Y^{k})$;}
\STATE {2.\ $S^{k+1}=\mathcal{D}_{a_{2},\lambda\mu^{-1}}(M-L^{k+1}+\mu^{-1}Y^{k})$;}
\STATE {3.\ $Y^{k+1}=Y^{k}+\mu(M-L^{k+1}-S^{k+1})$;}
\STATE {4.\ $k\rightarrow k+1$;}
\STATE{\textbf{end while}}
\STATE{\textbf{return}: $L^{\ast}$, $S^{\ast}$}
\end{algorithmic}
\end{algorithm}

It should be emphasized that the choice of penalty parameters $\mu$ and $\lambda$ have a great influence on the performance of Algorithm \ref{alg1}, and how to choose the best parameters $\mu$ and
$\lambda$ in Algorithm \ref{alg1} is a very hard problem. In this paper, we choose the parameter $\mu$ as
\begin{equation}\label{equ27}
\mu_{k+1}=\min\{\rho\mu_{k}, \bar{\mu}\},\ \ \ k=0, 1, 2, \cdots
\end{equation}
in each iteration in Algorithm \ref{alg1}, where $\rho>1$ is a constant factor and $\bar{\mu}$ is a given positive number.

In addition, the cross-validation method is accepted for the choice of the parameter $\lambda$ in Algorithm \ref{alg1}. We suppose that the matrix $S^{\ast}$ of sparsity $\gamma$ is the
optimal solution to the augmented Lagrange problem (\ref{equ10}). In each iteration, we rearrange the absolute value of elements of the matrix $M-L^{k+1}+\mu_{k}^{-1}Y^{k}\in \mathbb{R}^{m\times n}$
as a nonincreasing rearrangement vector by the Matlab code:
\begin{equation}\label{equ28}
h^{k+1}=sort(abs(T^{k+1}(:)),`descend'),
\end{equation}
where $T^{k+1}=M-L^{k+1}+\mu_{k}^{-1}Y^{k}$. By the operation (\ref{equ28}), we have $h^{k+1}_{1}\geq h^{k+1}_{2}\geq\cdots\geq h^{k+1}_{mn}$. Therefore, the following inequalities hold:
$$h^{k+1}_{i}>t_{a_{2},\lambda\mu_{k}^{-1}}\Leftrightarrow i\in \{1,2,\cdots,\gamma\},$$
$$h^{k+1}_{j}\leq t_{a_{2},\lambda\mu_{k}^{-1}}\Leftrightarrow j\in \{\gamma+1, \gamma+2,\cdots, n\},$$
where $t_{a_{2},\lambda\mu_{k}^{-1}}$ is the threshold value which is defined in Lemma \ref{lem1} which obtained by replacing $a$ and $\lambda$ with $a_{2}$ and $\lambda\mu_{k}^{-1}$ in $t_{a,\lambda}$.
According to $\sqrt{2\lambda\mu_{k}^{-1}}-\frac{1}{2a_{2}}\leq \lambda\mu_{k}^{-1}a_{2}$, we have
\begin{equation}\label{equ29}
\left\{
  \begin{array}{ll}
   h^{k+1}_{\gamma}>\sqrt{2\lambda\mu_{k}^{-1}}-\frac{1}{2a_{2}}; \\
   h^{k+1}_{\gamma+1}\leq\lambda\mu_{k}^{-1}a_{2},
  \end{array}
\right.
\end{equation}
which implies
\begin{equation}\label{equ30}
\frac{\mu_{k} h^{k+1}_{\gamma+1}}{a_{2}}\leq\lambda<\frac{\mu_{k}(2a_{2}h^{k+1}_{\gamma}+1)^{2}}{8a_{2}^{2}}.
\end{equation}
Therefore, in each iteration, a choice of parameter $\lambda$ in Algorithm \ref{alg1} can be selected as
\begin{equation}\label{equ31}
\lambda=\left\{
            \begin{array}{ll}
              \frac{\mu_{k} h^{k+1}_{\gamma+1}}{a_{2}}, & \ \ {\mathrm{if}\ \frac{\mu_{k} h^{k+1}_{\gamma+1}}{a_{2}}\leq\frac{\mu_{k}}{2a_{2}^{2}};} \\
              \frac{(1-\epsilon)\mu_{k}(2a_{2}h^{k+1}_{\gamma}+1)^{2}}{8a_{2}^{2}},  &\ \ {\mathrm{if}\ \frac{\mu_{k} h^{k+1}_{\gamma+1}}{a_{2}}>\frac{\mu_{k}}{2a_{2}^{2}},}
            \end{array}
          \right.
\end{equation}
where $\epsilon>0$ is a very small positive number such as 0.01 or 0.001. There is one more thing needed to be mentioned that the threshold value
$t_{a_{2},\lambda\mu_{k}^{-1}}=\lambda\mu_{k}^{-1}a_{2}$ if $\lambda=\frac{\mu_{k} h^{k+1}_{\gamma+1}}{a_{2}}$, and $t_{a_{2},\lambda\mu_{k}^{-1}}=\sqrt{2\lambda\mu_{k}^{-1}}-\frac{1}{2a_{2}}$
if $\lambda=\frac{(1-\epsilon)\mu_{k}(2a_{2}h^{k+1}_{\gamma}+1)^{2}}{8a_{2}^{2}}$.

When doing so, the Algorithm \ref{alg1} will be adaptive and free from the choice of parameter $\lambda$. By above operations, the algorithm \ref{alg1} varies the
parameters $\mu$ and $\lambda$ by iteration for solving the augmented Lagrange problem (\ref{equ10}) can be summarized in Algorithm \ref{alg2}.

\begin{algorithm}[h!]
\caption{: ADMM algorithm for solving the augmented Lagrange problem (\ref{equ10})}
\label{alg2}
\begin{algorithmic}
\STATE {\textbf{Initialize}: $S^{0}, Y^{0}\in \mathbb{R}^{m\times n}$, $a_{1}>0$, $a_{2}>0$, $\mu_{0}>0$, $\bar{\mu}>0$, $\lambda>0$, $\rho>1$, $\epsilon>0$;}
\STATE {$k=0$;}
\STATE {\textbf{while} not converged \textbf{do}}
\STATE {$Z^{k}=M-S^{k}+\mu_{k}^{-1}Y^{k}$;}
\STATE {Compute the SVD of $Z^{k}$ as}
\STATE {$Z^{k}=U^{k}\left[
       \begin{array}{c}
         \mathrm{Diag}(\sigma(Z^{k})) \\
         \mathbf{0}_{(m-n)\times n} \\
       \end{array}
     \right](V^{k})^{\top}$;}
\STATE {if\ $\mu_{k}^{-1}\leq\frac{1}{2a_{1}^{2}}$\ then}
\STATE \ \ \ \ {$t_{a_{1}, \mu_{k}^{-1}}=\mu_{k}^{-1}a_{1}$;}
\STATE {else}
\STATE \ \ \ \  {$t_{a_{1}, \mu_{k}^{-1}}=\sqrt{2\mu_{k}^{-1}}-\frac{1}{2a_{1}}$;}
\STATE {for\ $i=1:n$}
\STATE \ \ \ \ {1.\ $\sigma_{i}(Z^{k})>t_{a_{1},\mu_{k}^{-1}}$, then $\sigma_{i}(Z^{k+1})=g_{a_{1},\mu_{k}^{-1}}(\sigma_{i}(Z^{k}))$;}
\STATE \ \ \ \ {2.\ $\sigma_{i}(Z^{k})\leq t_{a_{1},\mu_{k}^{-1}}$, then $\sigma_{i}(Z^{k+1})=0$;}
\STATE {$L^{k+1}=U^{k}\left[
       \begin{array}{c}
         \mathrm{Diag}(\sigma(Z^{k+1})) \\
         \mathbf{0}_{(m-n)\times n} \\
       \end{array}
     \right](V^{k})^{\top}$;}
\STATE {$T^{k+1}=M-L^{k+1}+\mu_{k}^{-1}Y^{k}$;}
\STATE {Rearrange the absolute value of elements of the matrix $T^{k+1}$ as a nonincreasing rearrangement vector:}
\STATE {$h^{k+1}=sort(abs(T^{k+1}(:)),`descend')$;}
\STATE {if\ $\frac{\mu_{k} h^{k+1}_{\gamma+1}}{a_{2}}\leq\frac{\mu_{k}}{2a_{2}^{2}}$\ then}
\STATE \ \ \ \ {$\lambda=\frac{\mu_{k} h^{k+1}_{\gamma+1}}{a_{2}}$; $t_{a_{2}, \lambda\mu_{k}^{-1}}=\lambda\mu_{k}^{-1}a_{2}$;}
\STATE {else}
\STATE \ \ \ \ {$\lambda=\frac{(1-\epsilon)\mu_{k}(2a_{2}h^{k+1}_{\gamma}+1)^{2}}{8a_{2}^{2}}$; $t_{a_{2}, \lambda\mu_{k}^{-1}}=\sqrt{2\lambda\mu_{k}^{-1}}-\frac{1}{2a_{2}}$;}
\STATE {for\ $l=1:m$}
\STATE \ \ \ \ {for\ $j=1:n$}
\STATE \ \ \ \ {1.\ $|T^{k+1}_{l,j}|>t_{a_{2},\lambda\mu_{k}^{-1}}$, then $S^{k+1}_{l,j}=g_{a_{2},\lambda\mu_{k}^{-1}}(T^{k+1}_{l,j})$;}
\STATE \ \ \ \ {2.\ $|T^{k+1}_{l,j}|\leq t_{a_{2},\lambda\mu_{k}^{-1}}$, then $S^{k+1}_{l,j}=0$;}
\STATE {$S^{k+1}=[S^{k+1}_{l,j}]$;}
\STATE {$Y^{k+1}=Y^{k}+\mu_{k}(M-L^{k+1}-S^{k+1})$;}
\STATE {$\mu_{k+1}=\min\{\rho\mu_{k}, \bar{\mu}\}$;}
\STATE {$k\rightarrow k+1$;}
\STATE{\textbf{end while}}
\STATE{\textbf{return}: $L^{\ast}$, $S^{\ast}$}
\end{algorithmic}
\end{algorithm}

\section{Numerical experiments} \label{section4}

In this section, we present some numerical experiments for the problem of low-rank and sparse matrix decomposition to demonstrate the performances of the Algorithm \ref{alg2}.

In Algorithm \ref{alg2}, the most important implementation detail is the initial choice of the parameter $\mu$. In these numerical experiments, we set the initial value of the
parameter $\mu$ as
\begin{equation}\label{equ32}
\mu_{0}=\min\bigg\{\frac{2}{(0.99\|M\|_{2}+\frac{1}{2a_{1}})^{2}}, \frac{a_{1}}{0.99\|M\|_{2}}\bigg\}.
\end{equation}
Moreover, we also set $\rho=1.5$, $\epsilon=0.01$ and $\bar{\mu}=\mu_{0}\times 10^{7}$ in Algorithm \ref{alg2}. We take $m=n$, and generate $m\times m$ available real data $M$ using $M=L+S$,
where $L\in \mathbb{R}^{m\times m}$ and $S\in \mathbb{R}^{m\times m}$ are the true low-rank and sparse matrices that we wish to recover, respectively. Without loss of generality,
the low-rank matrix $L$ is generated by the following Matlab code:
$$L=1/m*rand(m,r)*rand(r,m),$$
where $r<m$. Therefore, the matrix $L\in \mathbb{R}^{m\times m}$ has rank at most $r$. The sparse matrix $S\in \mathbb{R}^{m\times m}$ is constructed by setting a proportion of
entries to be $\pm 1$ and the rest to be zeros. The number of the nonzero elements of the sparse matrix $S$ is set to $\|S\|_{0}=spr\times m\times m$, where $spr\in[0,1]$ is the sparsity ration.
In the experiments, the relative errors are respectively denoted by
\begin{equation}\label{equ33}
\left\{
  \begin{array}{ll}
    \hbox{$\mathrm{rel.err}(M)=\|M-L^{k+1}-S^{k+1}\|_{F}/\max\{1,\|M\|_{F}\}$;} \\
    \hbox{$\mathrm{rel.err}(L)=\|L-L^{k+1}\|_{F}/\max\{1,\|L\|_{F}\}$;} \\
    \hbox{$\mathrm{rel.err}(S)=\|S-S^{k+1}\|_{F}/\max\{1,\|S\|_{F}\}$.}
  \end{array}
\right.
\end{equation}
The stopping criterion is defined as $\mathrm{rel.err}(M)\leq 10^{-6}$ or the maximum iteration equation equals to $1000$. The initial matrices $S^{0}, Y^{0}\in \mathbb{R}^{m\times m}$ in
Algorithm \ref{alg2} are chosen as the zero matrices. These numerical experiments are all conducted on a personal computer (3.40GHz, 16.0GB RAM) with MATLAB R2015b.

\begin{table}[htbp]
	\centering
	\caption{Performance of Algorithm \ref{alg2} with different $a_{1}$ and $a_{2}$, $spr=0.15$.}\label{table1}
	\begin{tabular}{cccccccccc}
		\toprule  
		$a_{1}$&$a_{2}$&$m$&$r$&$\mathrm{rel.err}(M)$&$\mathrm{rel.err}(L)$ &$\mathrm{rank}(L)$&$\mathrm{rel.err}(S)$&$\|S\|_{0}$&Iteration $k$  \\
		\midrule  
		$1$&1&400&35&7.75e-07&3.47e-05&35&1.41e-06&24000&26\\
		 & &  &40&4.66e-07&3.06e-05&40&1.49e-06&24000&27\\
		 & &  &50&7.65e-07&2.36e-05&50&1.44e-06&24000&28\\
		\midrule  
		$5$&5&400&35&5.08e-07&2.88e-05&35&1.10e-06&24000&29\\
		 & &  &40&8.48e-07&4.01e-05&40&1.83e-06&24000&29\\
		 & &  &50&5.75e-07&2.11e-05&50&1.29e-06&24000&31\\
		\midrule  
		$10$&10&400&35&8.11e-07&3.99e-05&35&1.54e-06&24000&31\\
		 & &  &40&5.31e-07&2.18e-05&40&1.17e-06&24000&32\\
		 & &  &50&6.08e-07&2.09e-05&50&1.25e-06&24000&33\\
		\midrule  
		$50$&50&400&35&8.58e-07&3.26e-05&35&1.26e-06&24000&35\\
		 & &  &40&7.28e-07&2.56e-05&40&1.19e-06&24000&36\\
		 & &  &50&9.67e-07&1.12e-01&338&9.20e-03&24000&79\\
		\midrule  
		$80$&80&400&35&9.88e-07&1.26e-01&319&7.20e-03&24000&127\\
		 & &  &40&9.85e-07&1.00e-01&316&6.50e-03&24000&125\\
		 & &  &50&9.96e-07&6.69e-02&313&5.40e-03&24000&126\\
		\bottomrule  
	\end{tabular}
\end{table}

In order to implement Algorithm \ref{alg2}, we need to determine the parameters $a_{1}$ and $a_{2}$, which influences the behaviour of Algorithm \ref{alg2}. In the numerical
tests, we only take $a_{1}=a_{2}$ and test Algorithm \ref{alg2} on a series of low-rank and sparse matrix decomposition problems with different parameters $a_{1}$ and $a_{2}$,
and set $a_{1}=a_{2}=1,5,10,50,80$, respectively. In the numerical experiments, we only take $m=400$ and $spr=0.15$, and the results are shown in Table \ref{table1}. Comparing the performances
of Algorithm \ref{alg2} for low-rank and sparse matrix decomposition problems with different parameters $a_{1}$ and $a_{2}$, we can find that the parameter $a_{1}=a_{2}=1$ seems to be the
optimal strategy for Algorithm \ref{alg2} in our numerical experiments.

Next, we demonstrate the performance of the Algorithm \ref{alg2} on some low-rank and sparse matrix decomposition problems with different $m$ and $spr$ when we set $a_{1}=a_{2}=1$. We set $m$ to 500, 600, 700, 800,
and $spr$ to 0.20 and 0.25.  Numerical results of the Algorithm \ref{alg2} for the low-rank and sparse matrix decomposition problems are reported in Tables \ref{table2} and \ref{table3}. We can see that
the Algorithm \ref{alg2} with $a_{1}=a_{2}=1$ performs very well in separating the low-rank matrix and sparse matrix.

\begin{table}[htbp]
	\centering
	\caption{Performance of Algorithm \ref{alg2} with $spr=0.20$, $a_{1}=a_{2}=1$.}\label{table2}
    \begin{tabular}{cccccccc}
		\toprule  
		$m$&$r$&$\mathrm{rel.err}(M)$&$\mathrm{rel.err}(L)$ &$\mathrm{rank}(L)$&$\mathrm{rel.err}(S)$&$\|S\|_{0}$&Iteration $k$  \\
		\midrule  
		500&50&5.31e-07&4.21e-05&50&1.95e-06&50000&29\\
		   &60&7.31e-07&4.92e-05&60&2.83e-06&50000&30\\
		   &70&9.92e-07&2.22e-05&70&1.50e-06&50000&33\\
		\midrule  
		600&60&7.40e-07&4.90e-05&60&2.63e-06&72000&30\\
		   &70&5.77e-07&4.09e-05&70&2.30e-06&72000&31\\
		   &80&6.69e-07&2.37e-05&80&1.4509e-06&72000&32\\
		\midrule  
		700&70&4.75e-07&3.82e-05&70&1.80e-06&98000&30\\
		   &80&7.59e-07&3.12e-05&80&1.59e-06&98000&31\\
		   &90&8.75e-07&2.58e-05&90&1.49e-06&98000&32 \\
		\midrule  
		800&80&8.03e-07&6.50e-05&80&2.92e-06&128000&30\\
		   &90&7.13e-07&4.93e-05&90&2.52e-06&128000&31\\
		   &100&7.00e-07&4.33e-05&100&2.55e-06&128000&32\\
		\bottomrule  
	\end{tabular}
\end{table}

\begin{table}[htbp]
	\centering
	\caption{Performance of Algorithm \ref{alg2} with $spr=0.25$, $a_{1}=a_{2}=1$.}\label{table3}
    \begin{tabular}{cccccccc}
		\toprule  
		$m$&$r$&$\mathrm{rel.err}(M)$&$\mathrm{rel.err}(L)$ &$\mathrm{rank}(L)$&$\mathrm{rel.err}(S)$&$\|S\|_{0}$&Iteration $k$  \\
		\midrule  
		500&50&6.69e-07&3.60e-05&50&1.49e-06&62500&31\\
		   &60&8.81e-07&3.88e-05&60&1.93e-06&62500&32\\
		   &70&7.68e-07&4.88e-05&70&2.92e-06&62500&34\\
		\midrule  
		600&60&6.06e-07&3.52e-05&60&1.42e-06&90000&31\\
		   &70&9.88e-07&5.76e-05&70&3.21e-06&90000&32\\
		   &80&7.14e-07&2.99e-05&80&2.07e-06&90000&36\\
		\midrule  
		700&70&7.42e-07&3.62e-05&70&1.43e-06&122500&32\\
		   &80&8.47e-07&6.90e-05&80&3.41e-06&122500&33\\
		   &90&8.03e-07&5.41e-05&90&3.20e-06&122500&34\\
		\midrule  
		800&80&6.70e-07&5.50e-05&80&2.20e-06&160000&32\\
		   &90&9.60e-07&5.62e-05&90&2.64e-06&160000&33\\
		   &100&7.84e-07&3.48e-05&100&1.82e-06&160000&34\\
		\bottomrule  
	\end{tabular}
\end{table}

\section{Conclusion}\label{section5}

In this paper, based on the nonconvex fraction function, we presented a nonconvex optimization model for the low-rank and sparse matrix decomposition problem. The ADMM algorithm is utilized
to solve our nonconvex optimization problem, and the numerical results on some low-rank and sparse matrix decomposition problems show that our method performs very well in recovering
low-rank matrices which are heavily corrupted by large sparse errors. Moreover, there are some interesting problems should be solved in our future work. First, the convergence of our
algorithm is not proved in this paper, and we would like to treat it as our future work. Second, we found that our algorithm is very sensitive to the choice of the parameters and how
to choose the best parameters for our algorithm is also a very hard problem for us at present, and we also would like to treat it as our future work.

\ifCLASSOPTIONcaptionsoff
  \newpage
\fi

\end{document}